\documentclass[11pt]{amsart}
\usepackage[a4paper]{geometry}
\usepackage{amsmath,amssymb,ascmac,amsthm,caption}
\usepackage{graphicx}
\usepackage{cases}
\usepackage{tcolorbox}
\tcbuselibrary{breakable, skins, theorems}

\theoremstyle{plain}
\newtheorem{thm}{Theorem}[section]
 \newtheorem{cor}[thm]{Corollary}
 \newtheorem{lem}[thm]{Lemma}
 \newtheorem{prop}[thm]{Proposition}

\theoremstyle{definition}
 \newtheorem{defn}[thm]{Definition}
 
 \theoremstyle{remark}
 
 \newtheorem*{ex}{Example}
 
 \newcommand{\paper}[6]{%
	 #1, {\it #2}, #3 {\bf #4} (#5), #6.			
}

\renewcommand{\proofname}{\textbf{Proof}}
\numberwithin{equation}{section}
\numberwithin{figure}{section}

\newcommand{\Ker}{\operatorname{Ker}}
\newcommand{\im}{\operatorname{Im}}

\newcommand{\rank}{\operatorname{rank}}
\newcommand{\hess}{\operatorname{Hess}}

\begin{document}
\title{Geometry on ruled surfaces with finite multiplicity}
\author{Hiroyuki Hayashi}
\maketitle

\begin{abstract}
We consider ruled surfaces with finite multiplicity. We study behaviors of the striction curves and the singularities of the ruled surfaces. We also give geometric meanings of invariants related to the ruled surfaces.
\end{abstract}

\section{Introduction}\label{sec1}
The study of ruled surfaces in $\mathbb{R}^3$ is a classical subject in differential geometry. Ruled surfaces have singularities in general. In this paper, we study pseudo-cylindrical ruled surfaces. 
In \cite{H2024}, the notion of the pseudo-cylindrical ruled surface which is not cylinder but have properties that is similar to cylinder is given.
In that paper, pseudo-cylindrical developable surfaces obtained from a space curve with a singular point are investigated.

It is known that developable surfaces are classified into cylinders, cones, and tangent developable surfaces and their gluing. Non-cylindrical developable surfaces are one of tangential developable surfaces. Cuspidal edges, swallowtails, and cuspidal cross caps appear on non-cylindrical developable surfaces as singular points. On the other hand, cuspidal beaks and Scherbak surfaces are singularities that do not appear in the non-cylindrical developable surfaces but appear in the tangent developable surfaces which are not non-cylindrical developable surfaces. 

In this paper, we study pseudo-cylindrical developable surfaces because tangent developable surfaces are included in pseudo-cylindrical developable surfaces.
To study pseudo-cylindrical developable surfaces, we consider behaviors of the striction curves on pseudo-cylindrical developable surfaces, and the conditions of the above singularities.


\section{Preliminaries}\label{sec2}
\subsection{Singularities of surface with a frontal and a wave front}
A map $f:(\mathbb{R}^2,0)\to\mathbb{R}^3$ is called a {\it frontal} if there exists a map $\nu:(\mathbb{R}^2,0)\to S^2$ such that for any $p\in((\mathbb{R}^2 ; (u,v)),0)$, it holds that $\langle\nu(p),f_u(p)\rangle=\langle\nu(p),f_v(p)\rangle=0$, where $S^2$ is the unit sphere in $\mathbb{R}^3$ and $\langle,\rangle$ is an inner product. We call $\nu$ a {\it unit normal vector field} of $f$.  We set a map $L=(f,\nu):(\mathbb{R}^2,0)\to \mathbb{R}^3\times S^2$. A frontal $f$ is a {\it wave front} if $L$ is an immersion. We say that $C^{\infty}$-map germs $f_i:(\mathbb{R}^2,x_i)\to(\mathbb{R}^3,y_i)(i=1,2)$ are $\mathcal{A}$-equivalent if there exist diffeomorphism germs $\phi:(\mathbb{R}^2,x_1)\to(\mathbb{R}^2,x_2)$ and $\psi:(\mathbb{R}^3,y_1)\to(\mathbb{R}^3,y_2)$ such that $\psi \circ f_1=f_2 \circ \phi$. A singular point $p_0$ of a map $f$ is called
\begin{itemize}
\item
a {\it cuspidal edge} if $f$ at $p_0$ is $\mathcal{A}$-equivalent to $(u,v)\mapsto(u,v^2,v^3)$ at $0$,
\item
a {\it swallowtail} if $f$ at $p_0$ is $\mathcal{A}$-equivalent to $(u,v)\mapsto(u, u  v^2 + 3 v^4, 2 u  v + 4 v^3)$ at $0$,
\item
a {\it cuspidal butterfly} if $f$ at $p_0$ is $\mathcal{A}$-equivalent to $(u,v)\mapsto(u,5v^4+2uv,4v^5+uv^2-u^2)$ at $0$,
\item
a {\it cuspidal cross cap} if $f$ at $p_0$ is $\mathcal{A}$-equivalent to $(u,v)\mapsto(u,v^2,uv^3)$ at $0$,
\item
a {\it cuspidal lips} if $f$ at $p_0$ is $\mathcal{A}$-equivalent to $(u,v)\mapsto(u,-2v^3-u^2v,3v^4-u^2v^2)$ at $0$,
\item
a {\it cuspidal beaks} if $f$ at $p_0$ is $\mathcal{A}$-equivalent to $(u,v)\mapsto(u,-2v^3+u^2v,3v^4-u^2v^2)$ at $0$,
\end{itemize}
Maps with cuspidal edge, swallowtail, cuspidal butterfly, cuspidal lips, or cuspidal beaks are a wave front. A map with cuspidal cross cap is a frontal but is not a wave front. Let $f:(\mathbb{R}^2,0)\to(\mathbb{R}^3,0)$ be a frontal with a unit normal vector field  $\nu$. For a coordinate system $(u,v)$ on $(\mathbb{R}^2,0)$, we define a function $\lambda:(\mathbb{R}^2,0)\to\mathbb{R}$ by $\lambda=\det(f_u,f_v,\nu)$ and call it the {\it signed area density of $f$}. A singular point $p_0\in (\mathbb{R}^2,0)$ is called a {\it non-degenerate singular point} if it holds that $d\lambda(p_0)\ne0$. In this case, there exists a smooth parameterization $\gamma(t):(\mathbb{R},0)\to (\mathbb{R}^2,0)$ of $S(f)$. Moreover, there exists a vector field $\eta$ on $(\mathbb{R}^2,0)$ such that $\langle\eta(p_0)\rangle_{\mathbb{R}}=\Ker df_{p_0}$ for any $p_0\in S(f)$, where $S(f)$ is the set of the singular points of $f$. We call $\eta$ a {\it null vector field}. Now we define a function $\phi_f(t)$ on $\gamma$ by
\[
\phi_f(t)=\det\left(d(f\circ\gamma)/dt(t),(\nu\circ\gamma)(t),d\nu(\eta(t))\right).
\]
The following lemma is well-known\cite{FSUY2008, SUY2009, IS2010}.
\begin{lem}\label{cond-sing}
Let $f:(\mathbb{R}^2,0)\to(\mathbb{R}^3,0)$ be a wave front and $p_0\in(\mathbb{R}^2,0)$ be a non-degenerate point of $f$. Then the singular point $p_0$ is 
\begin{itemize}
\item
a cuspidal edge if and only if $\eta\lambda(p_0)\ne0$,
\item
a swallowtail if and only if $\eta\lambda(p_0)=0$ and $\eta\eta\lambda(p_0)\ne0$,
\item
a cuspidal butterfly if and only if $\eta\lambda(p_0)=\eta\eta\lambda(p_0)=0$ and $\eta\eta\eta\lambda(p_0)\ne0$.
\end{itemize}
Let $f:(\mathbb{R}^2,0)\to(\mathbb{R}^3,0)$ be a frontal and $p_0\in(\mathbb{R}^2,0)$ be a non-degenerate point of $f$. The singular point $p_0$ is a cuspidal cross cap if and only if $\eta\lambda(p_0)\ne0$, $\phi_f(p_0)=0$, and  $\phi_f'(p_0)\ne0$.
Let $f:(\mathbb{R}^2,0)\to(\mathbb{R}^3,0)$ be a wave front and $p_0\in(\mathbb{R}^2,0)$ be a degenerate point of $f$. The singular point $p_0$ is 
\begin{itemize}
\item
a cuspidal lips if and only if $\rank(df)_{p_0}=1$ and $\det\hess\lambda(p_0)>0$,
\item
a cuspidal beaks if and only if $\rank(df)_{p_0}=1$, $\det\hess\lambda(p_0)<0$, and $\eta\eta\lambda(p_0)\ne0$.
\end{itemize}
\end{lem}

\subsection{Ruled surfaces and developable surfaces}\label{subsec2.1}
We review the notions and basic properties of ruled surfaces and developable surfaces. Let $\gamma:(\mathbb{R},0)\to(\mathbb{R}^3,0)$ and $\xi:(\mathbb{R},0)\to\mathbb{R}^3\backslash\{\mathbf{0}\}$ be analytic $C^{\infty}$-maps. Then we define a {\it ruled surface} $F:(\mathbb{R},0)\times \mathbb{R} \to \mathbb{R}^3$ by
\[
F(x,t)=\gamma(x)+t\xi(x).
\]
We call the map $\gamma$ a {base curve} and the map $\xi$ a {\it director curve}. If we fix a parameter $t_0\in\mathbb{R}$, the line $F(x,t_0)$ is called a {\it ruling} on $t_0$. We call a ruled surface with vanishing Gaussian curvature on the regular part a {\it developable surface}. If the direction of the director curve $\xi$ is constant, we call $F$ a {\it cylinder}. 

We set $\bar\xi(x)=\xi(x)/||\xi(x)||$ and $\bar{F}(x,t)=\gamma(x)+t\bar\xi(x)$. Then the image $\im F$ is equal to the image $\im\bar{F}$. Therefore, we may assume $\xi:(\mathbb{R},0)\to S^2$. Then it is known that a ruled surface $F$ is a developable surface if and only if 
\begin{equation}\label{dev-cond}
\det(\gamma'(x),\xi(x),\xi'(x))=0,
\end{equation}
for any $x\in(\mathbb{R},0)$ where $(\ )'=d/dx$ (cf., \cite{IT2003}). We say that $F$ is {\it non-cylindrical} if it holds that $\xi'(x)\ne\mathbf{0}$ for any $x\in (\mathbb{R},0)$. We define a curve $s(x)=F(x,t(x))$. Then $s(x)$ is called a {\it striction curve} if it satisfies that
\begin{equation}\label{st-cond}
\langle s'(x),\xi'(x)\rangle=0
\end{equation}
for any $x\in (\mathbb{R},0)$. If there exists a striction curve and the striction curve is constant, we call $F$ a {\it cone}. We assume that $F$ is non-cylindrical. The striction curve is given by
\[
s(x)=\gamma(x)-\frac{\langle\gamma'(x),\xi'(x)\rangle}{\langle\xi'(x),\xi'(x)\rangle}\xi(x).
\]
It is known that a singular point of the non-cylindrical ruled surface is located on the striction curve. Moreover the set of singular points of a non-cylindrical developable surface coincides with the striction curve\cite{IT2001}. It is known that developable surface $F$ is a frontal. It is known (cf., \cite{IT2003}) that a non-cylindrical developable surface $F$ is a wave front if and only if
\begin{equation}\label{dev-front-cond}
\psi(x)=\det\left(\xi(x),\xi'(x),\xi''(x)\right)\ne0.
\end{equation}
Let $F$ be a non-cylindrical developable surface. Then by \eqref{dev-cond}, there exist functions $\alpha, \beta:(\mathbb{R},0)\to\mathbb{R}$ such that $\gamma'(x)=\alpha(x)\xi(x)+\beta(x)\xi'(x)$. The striction curve of $F$ is written as $s(x)=\gamma(x)-\beta(x)\xi(x)$ and we see that $\lambda(x,t)=||\xi'(x)||(t+\beta(x))$. Then a singular point of $F$ is always non-degenerate. By Lemma \ref{cond-sing}, it is known that the following Lemma (cf., \cite{ISTa2017}).
\begin{lem}
With the above notations, a singular point $(x,-\beta(x))$ of $F$ is 
\begin{itemize}
\item
a cuspidal edge if and only if $\psi(x)\ne0$ and $\beta'(x)-\alpha(x)\ne0$,
\item
a swallowtail if and only if $\psi(x)\ne0$, $\beta'(x)-\alpha(x)=0$, and $\beta''(x)-\alpha'(x)\ne0$,
\end{itemize}
\end{lem}
We remark that $\beta'''(x)-\alpha''(x)=0$ is not equivalent to $\eta\eta\eta\lambda=0$. By \cite{FSUY2008}, we obtain that the following lemma.
\begin{lem}
With the above notations, a singular point $(x,-\beta(x))$ of $F$ is a cuspidal cross cap if and only if $\psi(x)=0$, $\psi'(x)\ne0$, and $\beta'(x)-\alpha(x)\ne0$.
\end{lem}

\section{Pseudo cylindrical ruled surfaces}\label{sec3}
Non-cylindrical ruled surface and cylinder are classical and well studied. But these ruled surfaces have a property that the derivative of the directior curve $\xi'(x)$ is identically equal to $0$ or not. On the other hands, if we have $\xi'(x)=\tilde\xi(x)x^k (k\geq1)$, then the ruled surface is neither a cylinder nor a non-cylindrical ruled surface. We study this case.

\subsection{Fundamental properties of pseudo cylindrical ruled surfaces}\label{subsec3.1}
Let $g:(\mathbb{R},0)\to\mathbb{R}^3$ be a $C^{\infty}$-map. We say that $g$ is {\it of finite multiplicity} if there exist a integer $m\in\mathbb{Z}_{\geq0}$ and a $C^{\infty}$-map $\tilde g:(\mathbb{R},0)\to\mathbb{R}^3$ with the following property:
\[
g(x)=\tilde g(x)x^m, g(0)\ne\mathbf{0}.
\]
We define a pseudo cylindrical ruled surface. 
\begin{defn}[{\cite[Section 4.1]{H2024}}]
Let $\xi'(x)$ be of finite multiplicity. Then the ruled surface $F(x,t)=\gamma(x)+t\xi(x)$ is called a {\it $k$-th pseudo cylindrical} if 
\[
\xi'(x)=\tilde\xi(x)x^k, \tilde\xi(0)\ne\mathbf{0}
\]
holds for $k\in\mathbb{Z}_{\geq0}$.
\end{defn}
We remark that the notion of the pseudo-cylindrical ruled surface which is not cylinder but have properties that is similar to cylinder is given, and a $0$-th pseudo cylindrical ruled surface is a non-cylindrical ruled surface.
We take a unit vector field $\xi_d:(\mathbb{R},0)\to S^2$ by
\[
\xi_d(x)=\frac{\tilde\xi(x)}{||\tilde\xi(x)||}.
\]
We give a orthonormal frame $\{\xi(x),\xi_d(x),\xi(x)\times\xi_d(x)\}$ along $F$. We have the following Frenet-Serret type formula
\[
\left(
\begin{array}{c}
\xi'(x)\\
\xi_d'(x)\\
(\xi(x)\times\xi_d(x))'
\end{array}
\right)
=\left(
\begin{array}{ccc}
0&\delta(x)&0\\
-\delta(x)&0&\rho(x)\\
0&-\rho(x)&0
\end{array}
\right)
\left(
\begin{array}{c}
\xi(x)\\
\xi_d(x)\\
\xi(x)\times\xi_d(x)
\end{array}
\right),
\]
where 
\begin{equation}\label{delta-defn}
\delta(x)=\langle\xi'(x),\xi_d(x)\rangle=||\tilde\xi(x)||x^k,
\end{equation}
\begin{equation}\label{rho-defn}
\rho(x)=\langle\xi_d'(x),\xi(x)\times\xi_d(x)\rangle=\det(\xi(x),\xi_d(x),\xi_d'(x)).
\end{equation}
We define functions $p,q,r:(\mathbb{R},0)\to\mathbb{R}^3$ by
\begin{equation}\label{pqr-defn}
\gamma'(x)=p(x)\xi(x)+q(x)\xi_d(x)+r(x)(\xi(x)\times\xi_d(x)).
\end{equation}
Then it holds that
\[
\gamma''=(p'-q\delta)\xi+(p\delta+q'-r\rho)\xi_d+(q\rho+r')(\xi\times\xi_d),
\]
and we also have
\[
F_x\times F_t=-(q+t\delta)(\xi\times\xi_d)+r\xi_d.
\]
Thus if $F$ has a singularity at $(x_0,t_0)$, then it holds that $-(q(x_0)+t_0\delta(x_0))=r(x_0)=0$. On the regular parts of $F$, we define a unit normal vector $\nu=(F_x\times F_t)/|F_x\times F_t|$. Let $F$ be regular at $(x,t)$.
The coefficients of the first fundamental forms are given by 
\[
E=p(x)^2+(q(x)+t\delta(x))^2+r(x)^2,\ F=p(x),\ G=1.
\]
On the regular parts of $F$, the coefficients of the second fundamental form are given by
\[
L=\frac{pr\delta+(q'+t\delta')r-(q+t\delta)r'-((q+t\delta)^2+r^2)\rho}{|F_x\times F_t|},\ 
M=\frac{r\delta}{|F_x\times F_t|},\ N=0.
\]
Then we obtain that the Gaussian curvature and the mean curvature
\[
K=-\frac{r\delta}{|F_x\times F_t|^{3/2}},\quad
H=-\frac{-pr\delta+(q'+t\delta')r-(q+t\delta)r'-((q+t\delta)^2+r^2)\rho}{2|F_x\times F_t|^{3/2}}.
\]
Since $F$ is $k$-th pseudo cylindrical, $F$ is a developable surface if and only if $r(x)\equiv0$, where $\equiv$ stands for the equality holds identically. Let $F$ be a developable surface. Then we have $K\equiv0$ and
\[
H=\frac{\rho}{2|F_x\times F_t|}.
\]

\subsection{Singularities of pseudo cylindrical ruled surfaces}\label{subsec3.2}
We set maps $\tilde q, \tilde r:(\mathbb{R},0)\to\mathbb{R}^3$ and integers $Q,R\in\mathbb{Z}_{\geq0}$ such that
\[
q(x)=\tilde q(x)x^Q, \quad r(x)=\tilde r(x)x^R.
\]
We remark that we do not assume that $\tilde q(0)\ne0$, $\tilde r(0)\ne0$. We define the extended notion of the striction curve.
\begin{defn}
Let a curve on a ruled surface $F$ be $s(x)=F(x,t(x))$. Then $s(x)$ is called a {\it striction curve} if it satisfies that $\langle s'(x), \xi_d(x) \rangle \equiv0$.
\end{defn}
The above definition is the same as in \cite[Definition 3.1.]{LLW2021}.
We have that
\begin{align*}
\langle s'(x),\xi_d(x)\rangle
&=\langle \gamma'(x)+t'(x)\xi(x)+t(x)\xi'(x),\xi_d(x)\rangle\\
&=q(x)+t(x)\delta(x).
\end{align*}
If $\langle s'(x),\xi_d(x)\rangle \equiv0$, then it holds that
\begin{equation}\label{t(x)}
t(x)=-\frac{q(x)}{\delta(x)}=-\frac{\tilde q(x)}{||\tilde\xi(x)||}x^{Q-k}.
\end{equation}
We call $t(x)$ the {\it pre-striction function} of $F$. The striction curve is given by
\begin{equation}\label{st.cur}
s(x)=\gamma(x)-\frac{\tilde q(x)}{||\tilde\xi(x)||}x^{Q-k}\xi(x).
\end{equation}
By the equation \eqref{st.cur}, the striction curve $s(x)$ is well-defined at $x=0$ if and only if $Q\geq k$. 
Assume that $Q\geq k$. Then we have
\[
s'(x)=\left\{p(x)-\left(\frac{q(x)}{\delta(x)}\right)' \right\}\xi(x)+r(x)(\xi(x)\times\xi_d(x)).
\]
We define a set $V=\{(x,t)\in(\mathbb{R},0)\times\mathbb{R}\ |\ x\ne0\}$. Then we have the following proposition.
\begin{prop}
Let $Q\geq k$. Then
\begin{itemize}
\item
$F(S(F)\cap V)$ coincides with the striction curve of $F$ if and only if $r\equiv0$, 
\item
$F(S(F)\cap V)=\emptyset$ if and only if there exist a function $\tilde r$ and an integer $R$ such that $r(x)=\tilde r(x)x^R$, $\tilde r(0)\ne0$.
\end{itemize}
\end{prop}
\begin{proof}
If $r\equiv0$, $F$ is a developable surface. It is known that the set of singular points of a non-cylindrical developable surface coincides with the striction curve \cite{IT2001}. We assume that $F(S(F)\cap V)$ coincides with the striction curve of $F$. By the condition of singularities, it holds that $q(x)+t\delta(x)\equiv0$ and $r(x)\equiv0$ on the striction curve. We set $r(x)=\tilde r(x)x^R$, $\tilde r(0)\ne0$. Then it holds that $r(x)\ne0$ at $x\in V$. Thus $F(S(F)\cap V)=\emptyset$. We assume $F(S(F)\cap V)=\emptyset$. Then it holds that $r(x)\ne0$ on the striction curve.
\end{proof}
The case that satisfies $r\equiv0$ (that is, $F$ is a developable surface) is discussed in the next section.
We assume that there exist a function $\tilde r$ and an integer $R$ such that $r(x)=\tilde r(x)x^R$, $\tilde r(0)\ne0$. Then if $q(0)+t\delta(0)=0$, $F$ has a singularity on $x=0$. we divide the following three cases
\[
\rm{(i)}\  \delta(0)\ne0, \quad
\rm{(ii)}\ q(0)\ne0,\ \rm{and}\ \delta(0)=0, \quad
\rm{(iii)}\ q(0)=\delta(0)=0.
\]
If $F$ satisfies case (i), then $F$ is a non-cylindrical ruled surface, and there exists a striction curve $s(x)$ and $F$ has a singularity at $x=0$. If $F$ satisfies case (ii), then $F$ is regular on $x=0$ because it holds that $q(0)+t\delta(0)\ne0$ for any $t$. Therefore the striction curve approaches the ruling on $x=0$ as an asymptotic line. If $F$ satisfies case (iii), the ruling of $x=0$ is all singularities of $F$.

\begin{ex}
Assume that $\xi_1(x)=(1,x,0)$, $\xi_2(x)=(1,x^3,x^4)$,
\[
\gamma_1(x)=(0,0,x^2),\quad
\gamma_2(x)=(0,3x,2x^2),\quad
\gamma_3(x)=(0,\frac{3}{2}x^2,\frac{4}{3}x^3).
\]
We define ruled surfaces $g_1$, $g_2$, $g_3$ by
\[
g_1(x,t)=\gamma_1(x)+t\xi_1(x),\quad
g_2(x,t)=\gamma_2(x)+t\xi_2(x),\quad
g_3(x,t)=\gamma_3(x)+t\xi_2(x).
\]
Then $g_1$ is a $0$-th pseudo cylindrical ruled surface, that is, $g_1$ is a non-cylindrical ruled surface. The surface $g_1$ is an example which appears the Whitney umbrella singularity. The ruled surface $g_2$ and $g_3$ are $2$-nd pseudo ruled surface. The surface $g_2$ (respectively, $g_3$) is an example which is satisfied with case (ii) (respectively, case (iii)).
\end{ex}

\begin{figure}[ht]
\centering
\includegraphics[width=0.7\columnwidth]{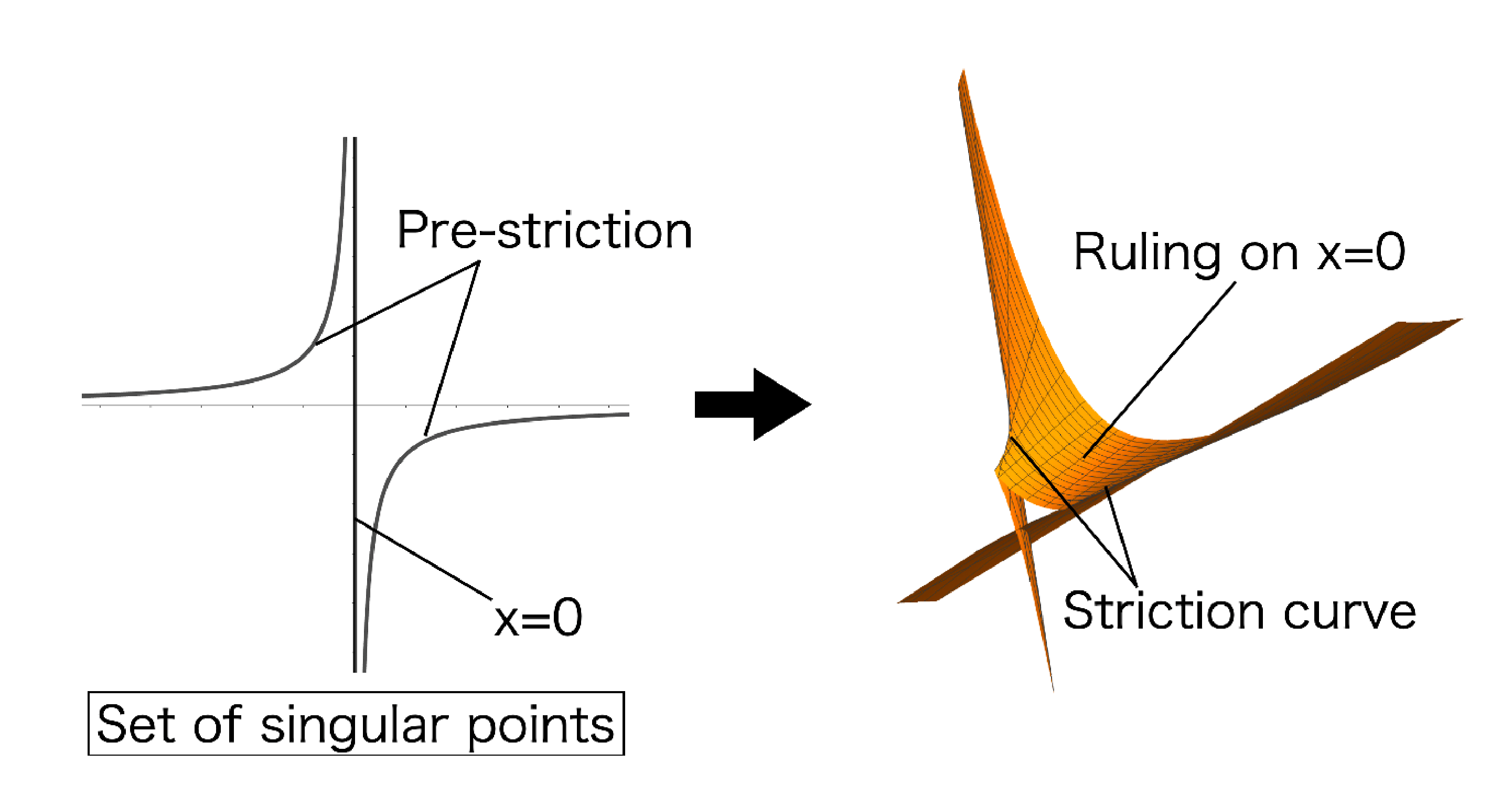}
\caption{Ruled surfaces  $g_2$}
\end{figure}

\begin{figure}[ht]
\centering
\includegraphics[width=0.7\columnwidth]{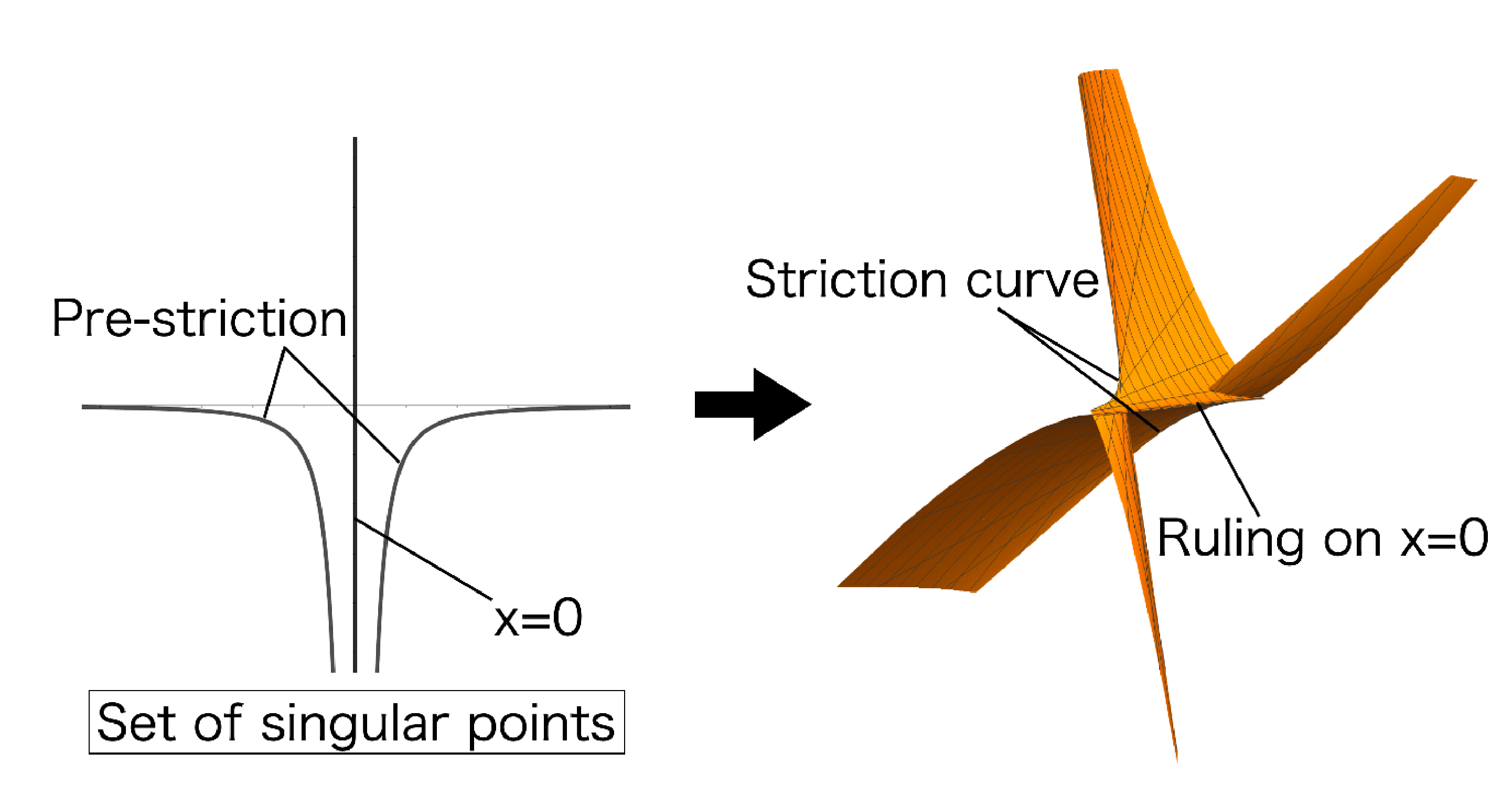}
\caption{Ruled surfaces $g_3$}
\end{figure}

Next, we obtain conditions of frontal and wave front of the ruled surface $F$.
Let $F$ have a singularities at $x=0$. Then it holds that $-(q(0)+t\delta(0))=r(0)=0$.  Let $m=\min\{Q,R\}$ be $k>m$, and set
\[
A(x,t)=-\tilde q(x)x^{Q-m}+t||\tilde\xi(x)|| x^{k-m},\quad B(x)=\tilde r(x)x^{R-m}.
\]
We obtain the unit normal vector of $F$:
\[
\nu(x,t)=\frac{A(x,t)(\xi\times\xi_d)+B(x)\xi_d}{\sqrt{A(x,t)^2+B(x)^2}}.
\]
Then $\nu(x,t)$ is well-defined on $x=0$. Thus we obtain the following lemma.
\begin{lem}
Suppose that $F$ is a ruled surface. If $k>\min\{Q,R\}$, then $F$ is a frontal.
\end{lem}
And we have the proposition.
\begin{prop}\label{ruled-wave}
Suppose that $F$ is a ruled surface and $k>\min\{Q,R\}$. Then $F$ is a wave front if and only if it holds that
\[
\rho(x)(A(x,t)^2+B(x)^2)+A_x(x,t)B(x)-A(x,t)B'(x)\ne0.
\]
\end{prop}
\begin{proof}
We have
\[
\left|
\begin{array}{cc}
F_x &\nu_x\\
F_t &\nu_t
\end{array}
\right|
=
\left|
\begin{array}{cc}
p\xi+(q+t\delta)\xi_d+r(\xi\times\xi_d) &\nu_x\\
\xi&\nu_t
\end{array}
\right|.
\]
On $x=0$, it holds that $\nu_t(0,t)=0$. Therefore, we have
\[
\nu_x(x,t)=\frac{-B(A^2+B^2)\delta\xi+\{\rho(A^2+B^2)+A_xB-AB'\}\{-A\xi_d+B(\xi\times\xi_d)\}}{(A(x,t)^2+B(x)^2)^{3/2}}.
\]
Since $\delta(0)=0$, it holds that $\nu_x\ne0$ if and only if it holds that $\rho(A^2+B^2)+A_xB-AB'$.
\end{proof}

We assume $k>\min\{Q,R\}$. Then it holds that $R\geq 1$ and $F$ is satisfied with case (iii). Thus we obtain $m=\min\{Q,R\}\geq1$. The signed area density $\lambda$ of $F$ is given by
\[
\lambda(x,t)=\det(F_x(x,t),F_t(x,t),\nu(x,t))=x^m\sqrt{A(x,t)^2+B(x)^2}.
\]
Then we obtain the following proposition.
\begin{prop}\label{ruled-deg}
Suppose that $k>\min\{Q,R\}$. A singular point $(0,t_0)$ is a non-degenerate point of $F$ if and only if $m=1$.
\end{prop}
\begin{proof}
We have $d\lambda=\lambda_xdx+\lambda_tdt$. Then it holds that 
\begin{flalign*}
&\lambda_x(x,t)=\frac{m(A(x,t)^2+B(x)^2)x^{m-1}+(A(x,t)A_x(x,t)+B(x)B'(x))x^m}{\sqrt{A(x,t)^2+B(x)^2}},\\
&\lambda_t(x,t)=\frac{A(x,t)A_t(x,t)x^m}{\sqrt{A(x,t)^2+B(x)^2}}.
\end{flalign*}
\end{proof}
A null vector $\eta$ of $F$ is given by
\begin{equation}\label{eta}
\eta=\partial x-p(x)\partial t.
\end{equation}
We have
\[
\eta\lambda(x,t)=\frac{m(A(x,t)^2+B(x)^2)x^{m-1}+(A(x,t)A_x(x,t)+B(x)B'(x)-p(x)A(x,t)A_t(x,t))x^m}{\sqrt{A(x,t)^2+B(x)^2}}.
\]
If $F$ is a wave front and $m=1$, then the singularity is a cuspidal edge.

\subsection{Singularities of pseudo cylindrical developable surfaces}\label{subsec3.2}
Let $F$ be a developable surface. Then it holds that 
\begin{align*}
F_x(x,t)\times F_t(x,t)= -(q(x)+t\delta(x))(\xi(x)\times\xi_d(x))
\end{align*}
We can write that there exists an integer $Q\in\mathbb{Z}_{\geq0}$ such that $q(x)=\tilde q(x)x^Q$, and if $k\geq Q$, then we assume that $\tilde q(x)\ne0$. We assume that $F$ has a singularity at the origin. Then $F$ satisfies that 
\[
\tilde q(x)x^Q+t||\tilde\xi(x)||x^k=0,
\] at $x=0$. For this equation, if $Q\geq k$, then we can define $t(x)=x^{Q-k}\tilde q(x)/ ||\tilde\xi(x)||$ on $x=0$, and if $Q,k\geq1$, then it holds that $x^m(\tilde q(x)x^{Q-m}+t||\tilde\xi(x)||x^{k-m})=0$ where $m=\min\{Q,k\}$.
In other words, if $\tilde q(x)x^Q+t||\tilde\xi(x)||x^k=0$, it holds that
\begin{itemize}
\item
if $Q\geq k$, there exist the siriction curve on the ruling on $x=0$.
\item
if $Q,k\geq1$, the ruling on $x=0$ is included in the singular point set.
\end{itemize}
The following table \ref{tab1} shows the conditions under which a $k$-th pseudo cylindrical developable surface has singularities on the ruling $x=0$, where $t(x)$ is the pre-striction function \eqref{t(x)}. 
\begin{table}[h]
\begin{center}
\caption{Condition of singular points}
\label{tab1}
\begin{tabular}{cccc}
\hline
& & ruling on $x=0$  & $t(x)$ ($x\to0$)\\
\hline
$\rm(\,I\,)$&$k>Q=0$    &regular   &$\infty$\\
$\rm(I\hspace{-.1em}I)$&$Q\geq k=0$    &regular\footnotemark[1]   &$-q(x)/\delta(x)$\\
$\rm(I\hspace{-.1em}I\hspace{-.1em}I)$&$k>Q\geq1$    &singularities &$\infty$\\
$\rm(I\hspace{-.1em}V)$&$Q\geq k\geq1$    &singularities   &$-q(x)/\delta(x)$\\
\hline
\end{tabular}
\end{center}
\end{table}
\footnotetext[1]{We assume $t\ne-q/\delta$. If $t=-q/\delta$, then F has a singularity on striction curve.}

We call a singularity a {\it Scherbak surface} if $f$ at $p_0$ is $\mathcal{A}$-equivalent to $(u,v)\mapsto(u,v^3+uv^2,12v^5+10uv^4)$ at $0$. We assume that there exists the striction curve (That is, we assume $Q\geq k$). If $F$ satisfies the case $\rm(I\hspace{-.1em}I)$, then the singular points set of $F$ coincides with the striction curve and may be cuspidal edge or swallowtail (Figure \ref{reg-st}). If $F$ satisfies the case $\rm(I\hspace{-.1em}V)$, then the singular points set of $F$ coincides with the striction curve and the ruling $x=0$ and may be cuspidal beaks, Schbak surface(Figure \ref{sing-st1} and \ref{sing-st2}).

\begin{figure}[ht]
\centering
\includegraphics[width=0.7\columnwidth]{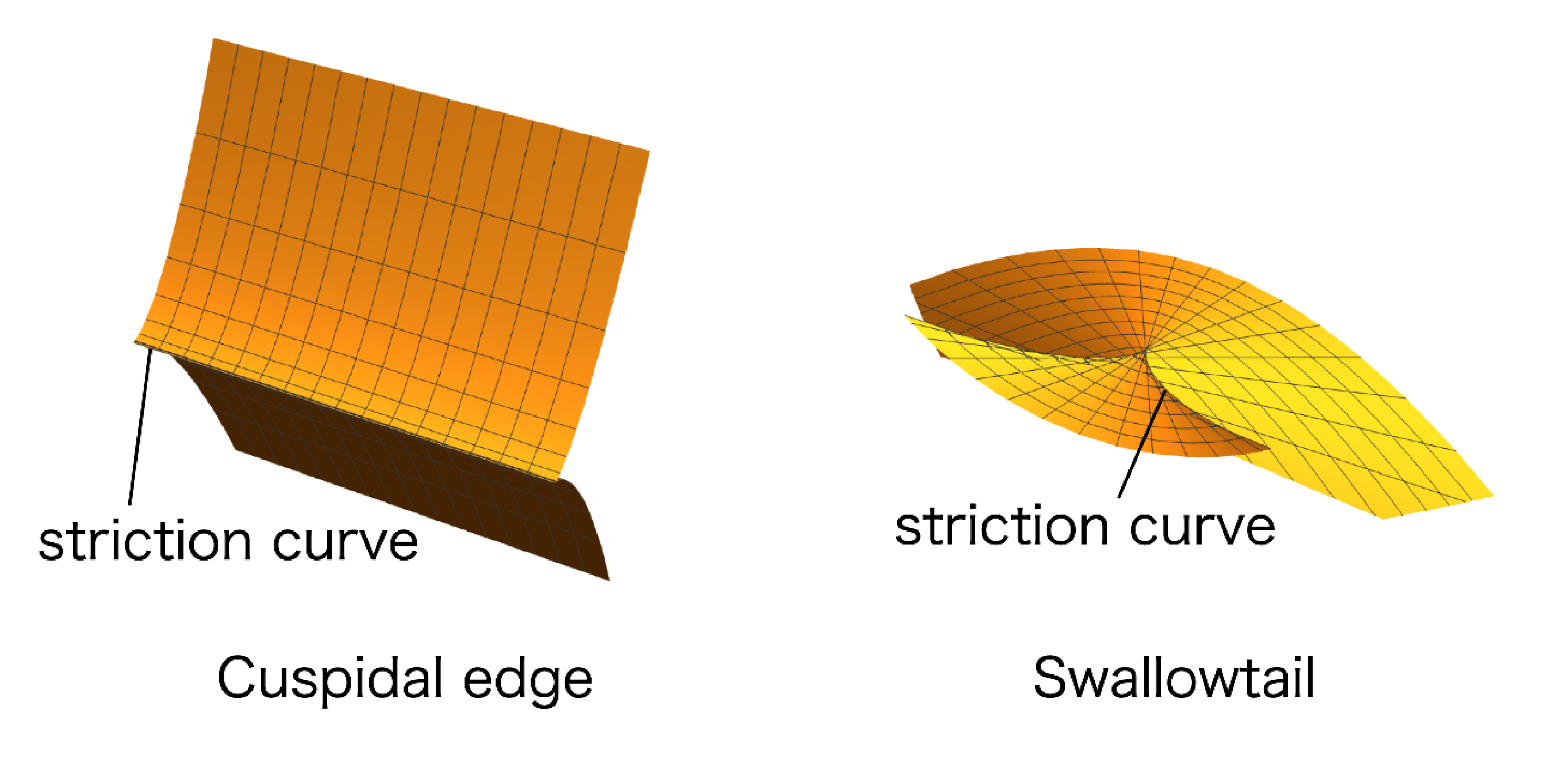}
\caption{Case $\rm(I\hspace{-.1em}I)$ (cuspidal edge (left) and swallowtail (right))}
\label{reg-st}
\end{figure}

\begin{figure}[ht]
\centering
\includegraphics[width=0.7\columnwidth]{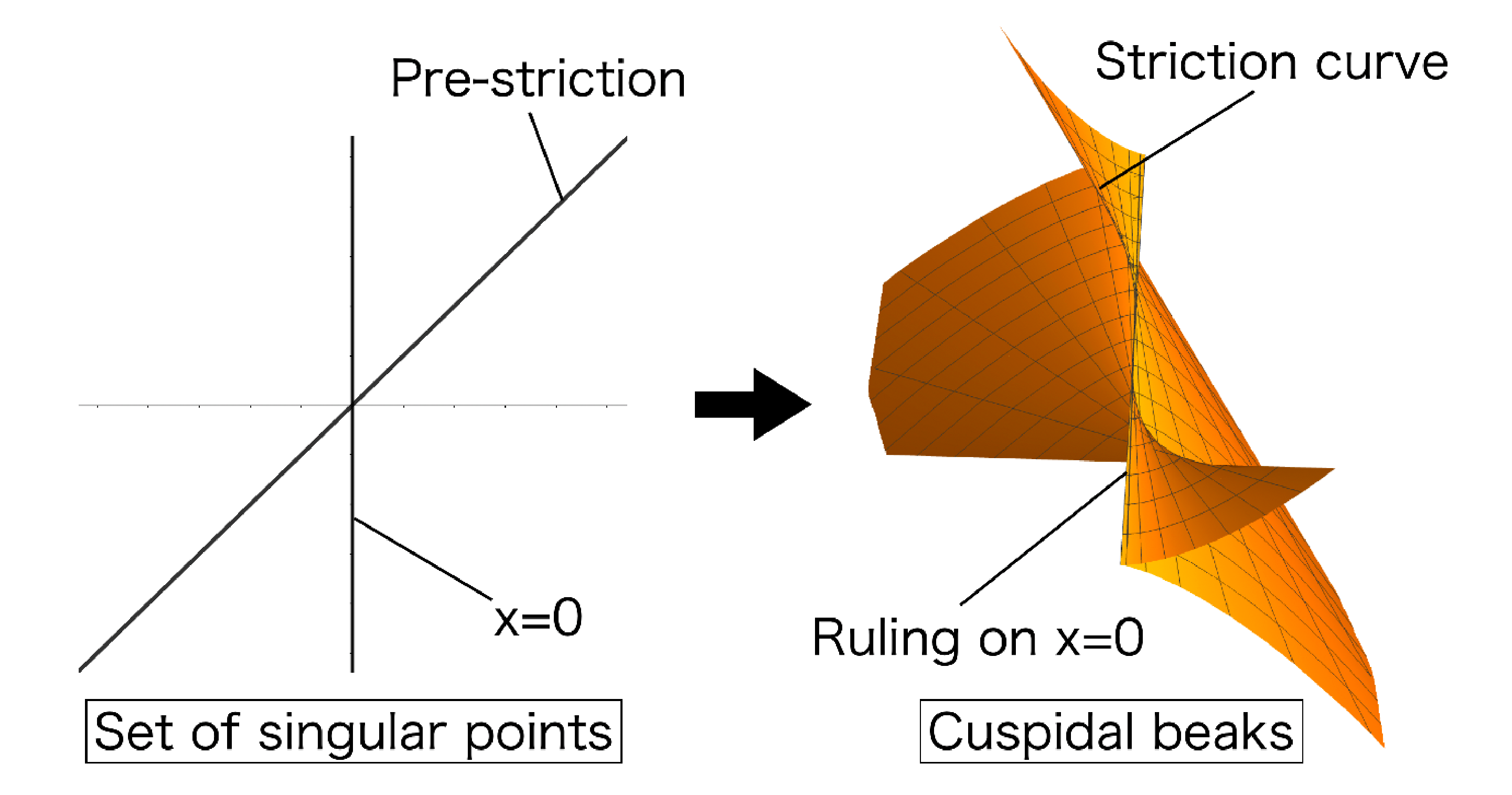}
\caption{Case $\rm(I\hspace{-.1em}V)$ (cuspidal beaks)}
\label{sing-st1}
\end{figure}

\begin{figure}[ht]
\centering
\includegraphics[width=0.7\columnwidth]{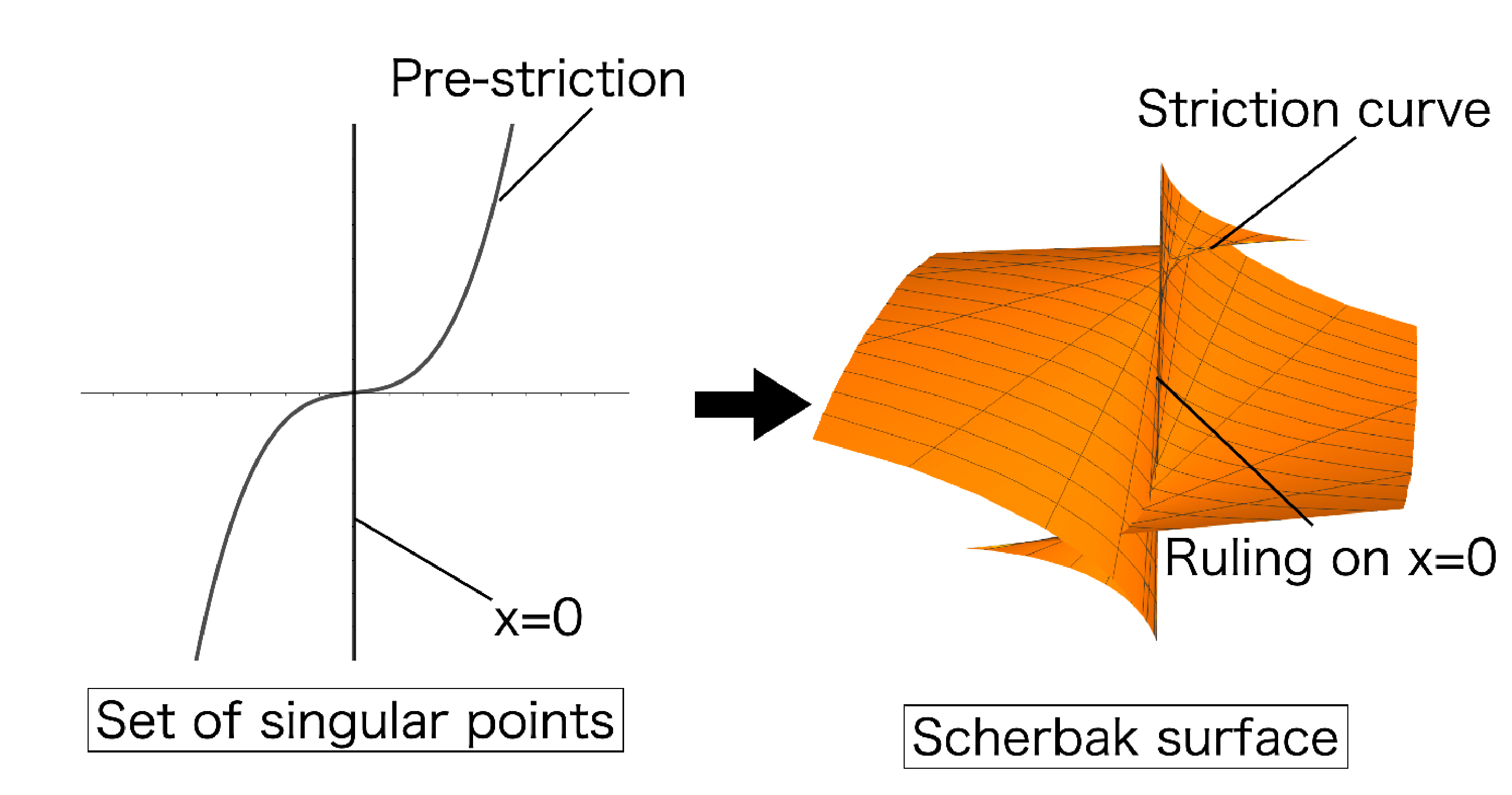}
\caption{Case $\rm(I\hspace{-.1em}V)$ (Scherbak surface)}
\label{sing-st2}
\end{figure}

\newpage

If $k>Q\geq1$, then the ruling on $x=0$ is included in the singular points set and the striction curve approaches the ruling on $x=0$ as an asymptotic line. We give an example of developable surface without striction curve on $x=0$.

\begin{ex}
Assume that $\xi_3(x)=(1,x^4,x^5)/\sqrt{1+x^8+x^{10}}$, $\xi_{3d}(x)=\xi_3'(x)/||\xi_3'(x)||$, and
\begin{align*}
\gamma_4(x)=\int_0^x\xi_{3d}(x)dx,\quad 
\gamma_5(x)=\int_0^xx\xi_{3d}(x)dx.
\end{align*}
We define ruled surfaces $g_1$, $g_2$;
\begin{equation}
g_4(x,t)=\gamma_1(x)+\xi(x),\quad
g_5(x,t)=\gamma_2(x)+\xi(x).
\end{equation}
Then $g_4$ (respectively, $g_5$) is a $3$-rd pseudo cylindrical developable surface and satisfies case $\rm(\,I\,)$ (respectively, a $3$-rd pseudo cylindrical developable surface and satisfies case $\rm(I\hspace{-.1em}I\hspace{-.1em}I)$ singularities).
\end{ex}

\begin{figure}[ht]
\centering
\includegraphics[width=0.7\columnwidth]{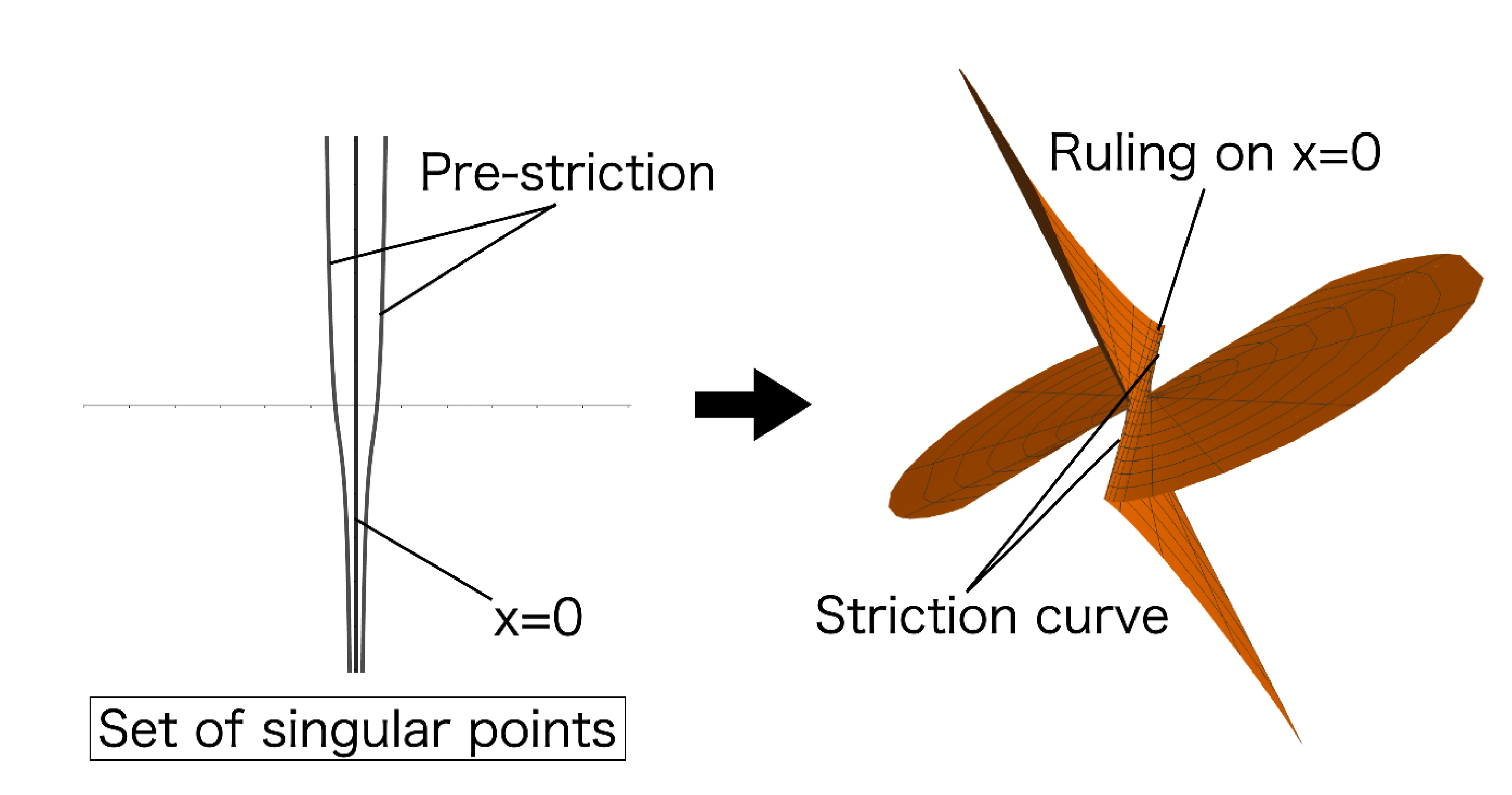}
\caption{Developable surface $g_4$ (case $\rm(\,I\,)$)}
\end{figure}

\begin{figure}[ht]
\centering
\includegraphics[width=0.7\columnwidth]{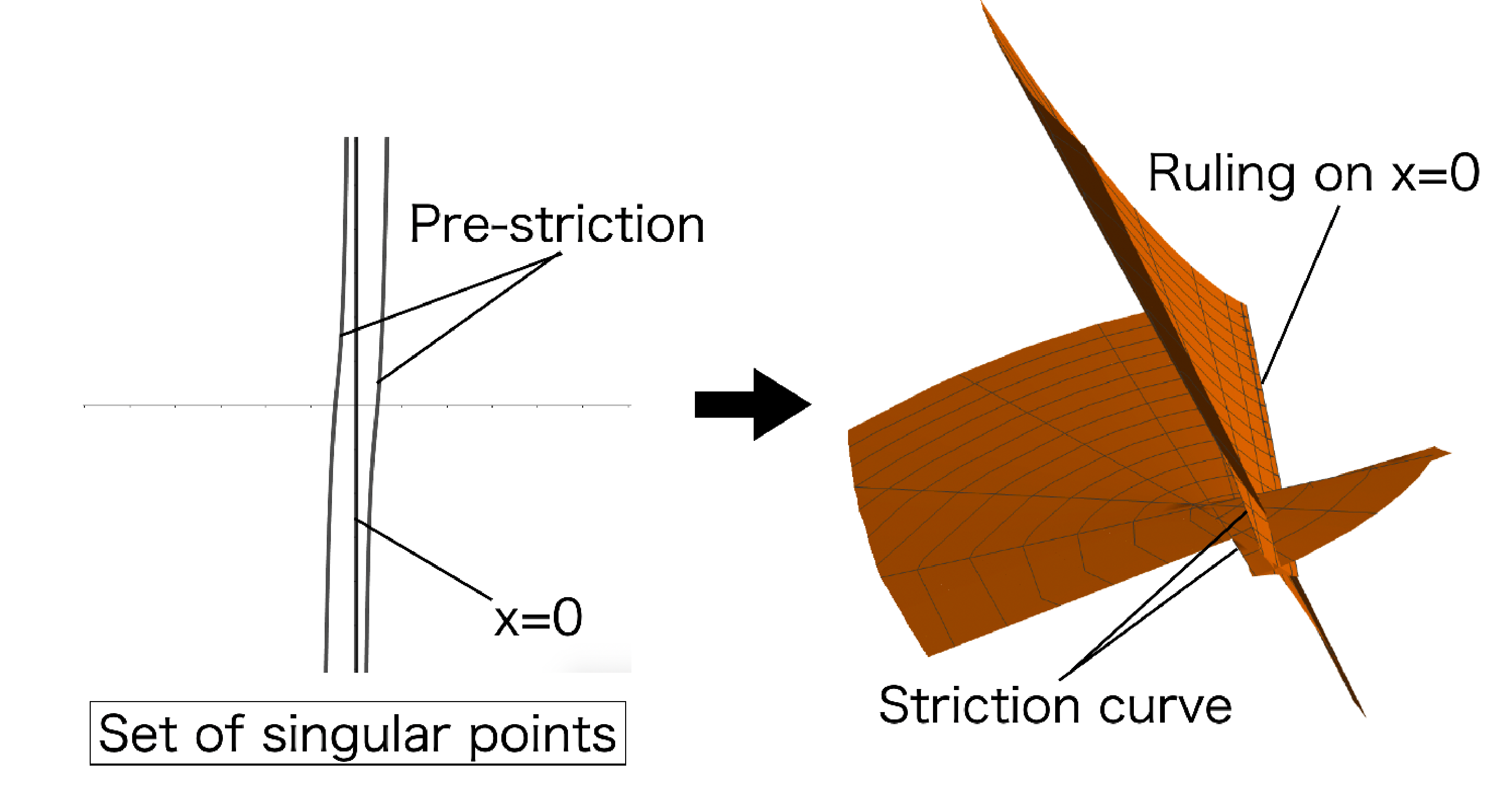}
\caption{Developable surface $g_5$ (case $\rm(I\hspace{-.1em}I\hspace{-.1em}I)$)}
\end{figure}

If $k>Q\geq1$, we obtain the following lemma
\begin{lem}
Suppose that $k>Q\geq1$. Then a point $p_0\in \{(x,t)\in(\mathbb{R}^2,0) | x=0\}$ is a non-degenerate point of $F$ if and only if $Q=1$. Therefore if $Q=1$, the singular point $p_0$ is a cuspidal edge.
\end{lem}
\begin{proof}
the signed area density of $F$ holds that 
\[
\lambda(x,t)=\det(F_x,F_t,\nu)=-(q(x)+t\delta(x)).
\]
We have
\[
d\lambda(x,t)=-(q'(x)+t\delta'(x))dx+\delta(x)dt.
\]
Since $k>Q\geq1$, $d\lambda(0,t)\ne0$ if and only if $Q=1$. And if $Q=1$, $\eta\lambda(0,t)\ne0$.
\end{proof}

\subsection{Geometric meanings of invariants $\delta$, $\rho$, and $\sigma$}
The invariants $\delta$, $\rho$, and $\sigma$ give the conditions of singularities.

We have a unit normal vector $\nu(x)=\xi(x)\times\xi_d(x)$. That is, a developable surface is a frontal. Then the signed area density of $F$ holds that 
\[
\lambda(x,t)=\det(F_x,F_t,\nu)=-(q(x)+t\delta(x)).
\]
We have
\[
d\lambda(x,t)=-(q'(x)+t\delta'(x))dx+\delta(x)dt.
\]
Suppose that $F$ has singularities on $x=0$. That is, it holds that $k>Q\geq1$ or $Q\geq k$. Then we have the following lemma.
\begin{prop}\label{deg-delta}
Suppose that $k>Q\geq1$. Then a point $(0,t)$ is a non-degenerate point of $F$ if and only if $Q=1$. Also, assume that $Q\geq k$. Then we have the following assertion.
\begin{itemize}
\item
A singular point $(0,-q(0)/\delta(0))$ is a non-degenerate singular point if and only if $\delta(0)\ne0$, 
\item
a singular point $(0,-q(0)/\delta(0))$ is a degenerate singular point if and only if $\delta(0)=0$.
\end{itemize}
\end{prop}
\begin{proof}
Let $k>Q\geq1$. Then it holds that
\begin{align*}
d\lambda(x,t)&=-(q'(x)+t\delta'(x))dx+\delta(x)dt\\
&=-(Q\tilde q(x)x^{Q-1}+\tilde q'(x)x^Q+t(k||\tilde\xi(x)||x^{k-1}+||\tilde\xi(x)||'x^k))dx+||\tilde\xi(x)||x^kdt.
\end{align*}
Since $k>Q$, it holds that $d\lambda(0,-q(0)/\delta(0))\ne0$ if and only if $Q=1$. Also, we assume that $Q\geq k$. Then 
\begin{align*}
d\lambda(x,t)&=-(Q\tilde q(x)x^{Q-1}+\tilde q'(x)x^Q+\frac{\tilde q(x)x^{Q}}{||\tilde\xi(x)||x^k}(k||\tilde\xi(x)||x^{k-1}+||\tilde\xi(x)||'x^k))dx+||\tilde\xi(x)||x^kdt\\
&=-((Q-k)\tilde q(x)x^{Q-1}+\frac{\tilde q'(x)||\tilde\xi(x)||+\tilde q(x)||\tilde\xi(x)||'}{||\tilde\xi(x)||}x^Q)dx+||\tilde\xi(x)||x^kdt
\end{align*}
Since $Q\geq k$, we have the above assertion.
\end{proof}

Next, the necessary and sufficient condition that $F$ is a wave front is given the following theorem.
\begin{prop}\label{wave-rho}
The developable surface $F$ is a wave front at $x=0$ if and only if $\rho(0)\ne0$.
\end{prop}
\begin{proof}
Let $L(x,t)=(F(x,t),\nu(x))$ be a $C^{\infty}$-map. Then the Jacobian matrix $J$ of $L$ is given by
\[
J(x,t)=\left|
\begin{array}{cc}
F_x & \nu_x \\
F_t & \nu_t
\end{array}
\right|
=\left|
\begin{array}{cc}
p(x)\xi(x)+(q(x)+t\delta(x))\xi_d(x) & \nu'(x)\\
\xi(x) & 0
\end{array}
\right|.
\]
The map $L$ is an immersion if and only if $J(x,t)=-\xi(x)\nu'(x)\ne0$. It holds  that
\begin{align*}
\nu'(x)&=(\xi(x)\times\xi_d(x))'\\
&=\delta(x)\xi_d(x)\times\xi_d(x)+\xi(x)\times\{-\delta(x)\xi(x)+\rho(x)(\xi(x)\times\xi_d(x))\}\\
&=-\rho(x).
\end{align*}
Since $\xi(0)\ne0$, $J\ne0$ if and only if $\rho\ne0$.
\end{proof}
We remark that if the developable surface $F$ is $0$-th pseudo cylindrical, then it satisfies $\rho(0)\ne0$ if and only if it holds the equation \eqref{dev-front-cond}.

A null vector $\eta$ of $F$ is given by \eqref{eta}.
Then it holds that
\[
\eta\lambda(x,t)=-(q'(x)+t\delta'(x))+p(x)\delta(x),
\]
\[
\eta\eta\lambda(x,t)=-(q''(x)+t\delta''(x))+p'(x)\delta(x)+2p(x)\delta'(x).
\]
Similarly calculating, it holds that
\[
\eta^{(n)}(x,t)=-(q^{(n)}(x)+t\delta^{(n)}(x))+\sum_{i=0}^{n-1} {}_n C_{i} p^{(n-i-1)}(x)\delta^{(i)}(x).
\]

If $k>Q\geq1$, then we have $\eta^{(Q)}(0,t)=-Q!\tilde q(0)$. Thus we obtain the following lemma.
\begin{lem}
Let $k>Q\geq1$ and $p_0=(0,t_0)$. Then for any $n\geq1$, it holds that
\begin{itemize}
\item
$\eta\lambda(p_0)\ne0$ if and only if $Q=1$,
\item
$\eta\lambda(p_0)=\cdots=\eta^{(n)}\lambda(p_0)=0,\ 
\eta^{(n+1)}\lambda(p_0)\ne0$ if and only if $Q=n+1$.
\end{itemize}
\end{lem}
Suppose that $Q\geq k$. Then we obtain the following lemma and theorem.
\begin{thm}\label{thm-reg.ver}
Let $Q\geq k$ and the striction curve be regular at $p_0=(0,-q(0)/\delta(0))$. Then for any $n\geq1$, it holds that
\begin{itemize}
\item
$\eta\lambda(p_0)\ne0$ if and only if $k=0$,
\item
$\eta\lambda(p_0)=\cdots=\eta^{(n)}\lambda(p_0)=0,\ 
\eta^{(n+1)}\lambda(p_0)\ne0$ if and only if $k=n$.
\end{itemize}
\end{thm}
\begin{proof}
Since the developable surface $F$ is $k$-th pseudo cylindrical, we obtain that
\[
\delta(0)=\delta'(0)=\cdots=\delta^{(k-1)}(0)=0, \delta^{(k)}(0)\ne0.
\]
Therefore since $Q\geq k$, we obtain that
\[
q(0)=q'(0)=\cdots=q^{(k-1)}(0)=0, q^{(k)}(0)\ne0.
\]
By the above equations, $\eta\lambda(0,-q(0)/\delta(0))=\cdots=\eta^{(k-1)}\lambda(0,-q(0)/\delta(0))=0$, and 
\begin{align*}
\eta^{(k)}(0,-q(0)/\delta(0))&=-(q^{(k)}(0)-t\delta^{(k)}(0))\\
&=-k!(\tilde q(0)-\frac{\tilde q(0)}{||\tilde\xi(0)||}||\tilde\xi(0)||)=0,
\end{align*}
\begin{align*}
\eta^{(k+1)}(0,-q(0)/\delta(0))&=-(q^{(k+1)}(0)-t\delta^{(k+1)}(0))+(k+1)p(x)\delta^{(k)}(0)\\
&=-(k+1)!\left(\tilde q'(0)-\frac{\tilde q(0)}{||\tilde\xi(0)||}||\tilde\xi(0)||'-p(0)||\tilde\xi(0)||\right)\\
&=-(k+1)!||\tilde\xi(0)||\left(\frac{\tilde q'(0)||\tilde\xi(0)||-\tilde q(0)||\tilde\xi(0)||'}{||\tilde\xi(0)||^2}-p(0)\right)\\
&=(k+1)!||\tilde\xi(0)||\left(p(0)-\left(\frac{q(0)}{\delta(0)}\right)'\right)
\end{align*}
Since the striction curve is regular at $0$, it holds that $p(0)-(q(0)/\delta(0))'\ne0$.
\end{proof}
\begin{thm}\label{thm-sing.ver}
Let $Q\geq k$. If the striction curve has a singularity at $0$, then for any $i\geq1$, it holds that
\[
\eta^{(n)}\lambda\left(0,-\frac{q(0)}{\delta(0)}\right)=\sigma^{(n-1)}(0).
\]
\end{thm}
\begin{proof}
When $j=1$, it holds that
\begin{align*}
\sigma(0)&=\delta(0)\left(p(0)-\left(\frac{q(0)}{\delta(0)}\right)'\right)\\
&=-\left(q'(0)-\frac{q(0)\delta'(0)}{\delta(0)}\right)+p(0)\delta(0)\\
&=\eta\lambda\left(0,\frac{q(0)}{\delta(0)}\right).
\end{align*}
When $j=n$, we assume that $\eta^{(n)}\lambda\left(0,-q(0)/\delta(0)\right)=\sigma^{(n-1)}(0)$. Then 
\begin{align*}
\sigma^{(n)}(0)&=\frac{d}{dx}\sigma^{(n-1)}(0)\\
&=-\left(q^{(n+1)}(0)-\left(\frac{q(0)}{\delta(0)}\right)'\delta^{(n)}(0)-\left(\frac{q(0)}{\delta(0)}\right)\delta^{(n+1)}(0)\right)\\
&\qquad +\sum_{i=0}^{n} {}_{n+1} C_{i} p^{(n-i)}(x)\delta^{(i)}(0)-p(0)\delta^{(n)}(0)\\
&=\eta^{(n+1)}\lambda\left(0,-\frac{q(0)}{\delta(0)}\right)-\left(p(0)-\left(\frac{q(0)}{\delta(0)}\right)'\right)\delta^{(n)}(0).
\end{align*}
Since the striction curve has a singularity at $0$, it holds that $p(0)-(q(0)/\delta(0))'=0$.
\end{proof}
By Theorem \ref{thm-reg.ver} and Theorem \ref{thm-sing.ver}, we obtain the following Corollary.
\begin{cor}\label{eta-st}
Let $Q\geq k$. Then for any $n\geq1$, it holds that
\begin{itemize}
\item
$\eta\lambda\ne0$ if and only if $\sigma\ne0$,
\item
$\eta\lambda=\cdots=\eta^{(n)}\lambda=0, \eta^{(n+1)}\lambda\ne0$ if and only if $F$ is satisfied with $\sigma=\cdots=\sigma^{(n-1)}=0, \sigma^{(n)}\ne0$.
\end{itemize}
\end{cor}

Suppose that there exists a striction curve of $F$ for any $x$. By the proposition \ref{deg-delta}, \ref{wave-rho}, and the Corollary \ref{eta-st}, a singular point $(x,-q(x)/\delta(x))$ of $F$ is
\begin{itemize}
\item
a cuspidal edge if and only if $\delta(x)\ne0$, $\rho(x)\ne0$, and $\sigma(x)\ne0$,
\item
a swallowtail if and only if $\delta(x)\ne0$, $\rho(x)\ne0$, $\sigma(x)=0$, and $\sigma'(x)\ne0$,
\item
a cuspidal butterfly if and only if $\delta(x)\ne0$, $\rho(x)\ne0$, $\sigma(x)=\sigma'(x)=0$, and $\sigma''(x)\ne0$.
\item
a cspidal cross cap if and only if $\delta(x)\ne0$, $\rho(x)=0$, $\rho'(x)\ne0$, and $\sigma(x)\ne0$.
\end{itemize}
And it holds that $\hess \lambda=-(\delta')^2$. If $\delta(x)=0$, $\delta'(x)\ne0$, $\hess\lambda<0$. Then we obtain the following corollary.
\begin{cor}
Soppose that there exists a striction curve of $F$ for any $x$. Then a singular point $(x,-q(x)/\delta(x))$ of $F$ is a cuspidal beaks if and only if $\delta(x)=0$, $\delta'(x)\ne0$, $\rho(x)\ne0$ and $\sigma'(x)\ne0$.
\end{cor}

\end{document}